\documentclass[11pt,a4paper]{article}
\usepackage[pdftex]{graphicx}
\usepackage{amsmath,amssymb,amsfonts,amsthm}
\numberwithin{equation}{section}
\usepackage{indentfirst}
\usepackage{enumitem} % for customizing lists
\usepackage[margin=1.3in]{geometry}
\usepackage{fancyhdr}
\usepackage{makecell}

\usepackage[colorlinks=true,linkcolor=magenta,citecolor=blue,urlcolor=cyan]{hyperref}
\usepackage{titlesec}
\usepackage{etoolbox}
\usepackage{graphicx}  % for scalebox if needed later
\usepackage{changepage}
\theoremstyle{plain}
\newtheorem{theorem}{Theorem}[section]
\newtheorem{lemma}[theorem]{Lemma}
\newtheorem{corollary}[theorem]{Corollary}

\newtheorem{definition}[theorem]{Definition}

\makeatletter
\renewcommand{\maketitle}{
	\begin{center}
		{\Large\bfseries{\@title}\par}
		\vskip 1em
		{\normalsize
			\lineskip .5em
			\begin{tabular}[t]{c}
				\@author
			\end{tabular}\par}
		\vskip 1.5em
	\end{center}
}
\makeatother
\renewenvironment{abstract}{
	\begin{adjustwidth}{1.3cm}{1.3cm}
		\noindent{\large\bfseries{A{\scriptsize BSTRACT.}}}
	}{
	\end{adjustwidth}
}

\usepackage{color}
\usepackage{xcolor}
\usepackage[normalem]{ulem} 
\usepackage{soul}

\usepackage{graphicx} % already included
\usepackage{xstring}  % for string manipulation

\usepackage{xstring}
\usepackage{graphicx} % for \scalebox

\usepackage{marvosym}
\begin{document}
	
	\title{Arithmetic Properties of Overcolored Odd Partitions}

	\author{M. P. Thejitha and S. N. Fathima}
	
	\maketitle
	
	\begin{abstract}
		Let $\bar{a}_s(n)$ denote the number of partitions of $n$, wherein each odd part is multicolored (atmost $s\ge 1$ colors) and the first appearance of parts may be overlined. In this paper, we establish new families of congruences modulo powers of $2$ satisfied by $\bar{a}_s(n)$ for infinitely many $s$. Our approach builds upon generating function manipulations, Hecke eigenform theory and results of Newman.\\
		
		\noindent {\bf \small Keywords:} Partitions, Colored partitions, Overpartitions, Congruences, Modular forms\\
		
		\noindent {\bf \small Mathematics Subject Classification (2020):} 05A17, 11P83.
	\end{abstract}
	
	\bigskip
	
	\vspace{0.5em}
	
	\section{Introduction}
	
	For complex numbers $a$ and $q$ with $|q|<1$, we adopt the following standard notation from the theory of $q$-series \cite{gasper}:
	\begin{align*}
		(a;q)_\infty:=\prod_{k=0}^{\infty}(1-aq^k).
	\end{align*}
	For convenience, we set $f_\ell^m:=(q^\ell;q^\ell)^m_\infty$, for some integers $\ell,m\ge 1$.\\
	\indent We begin by recalling that Corteel and Lovejoy \cite{corteel} in $2004$ revisited the combinatorial object known as overpartitions. An overpartition of a positive integer $n$ is a partition of $n$ where the first occurrence of each part may be overlined. For instance, $2$ has the following $4$ overpartitions:
	\begin{align*}
		(2), (\bar{2}), (1,1), (\bar{1},1).	
	\end{align*}
	We denote the number of overpartitions of $n$ by $\bar{p}(n)$, and define $\bar{p}(0)=1$. Hence, $\bar{p}(2)=4$.
	As noted in \cite{corteel}, the generating function for $\bar{p}(n)$ is given by 
	\begin{align}\label{eoverp}
		\sum_{n=0}^{\infty}\bar{p}(n)q^n=\dfrac{f_2}{f_1^2}.
	\end{align}
	Recently, there are extensive literature works that delves deeper into the area of overpartitions.\\
	\indent	Let $p_2(n)$ denote the number of colored partitions of $n$, wherein each odd part appears in two different colors. At first glance, the generating function for $p_2(n)$ is defined as
	\begin{align}\label{e2color}
		\sum_{n=0}^{\infty}p_2(n)q^n=\dfrac{1}{(q;q^2)^2_\infty (q^2;q^2)_\infty}.
	\end{align}
	It is worth noting that $p_2(n)$ arises naturally in theory of integer partitions. Indeed, using Euler's identity
	\begin{align*}
		(-q,q)_\infty= \dfrac{1}{(q;q^2)_\infty},
	\end{align*} 
	we obtain 
	\begin{align}\label{e1}
		\sum_{n=0}^{\infty}p_2(n)q^n=\dfrac{(-q,q)_\infty}{(q;q^2)_\infty (q^2;q^2)_\infty}=\dfrac{(-q,q)_\infty (q;q)_\infty}{(q;q)_\infty (q;q)_\infty}=\sum_{n=0}^{\infty}\bar{p}(n)q^n.	\end{align}
	Clearly, \eqref{e1} establishes the equivalence between $p_2(n)$ and $\bar{p}(n)$. Andrews and Mohamed El Bachraoui \cite{andrew} have systematically explored connections between  $p_2(n)$ and $\bar{p}(n)$.\\
	\indent Given a positive integer $s$, define the partition function $\bar{a}_s(n)$, which counts the number of partitions of $n$ wherein each odd part occurs in $s$ different colors and the first occurrence of parts may be overlined. The generating function of $\bar{a}_s(n)$ is
	\begin{align}\label{gf}
		\sum_{n=0}^{\infty}\bar{a}_s(n)q^n=\dfrac{f_2^{3s-2}}{f_1^{2s}f_4^{s-1}}.
	\end{align}
	The partition function  $\bar{a}_{s}(n)$ is a natural extension that arise by considering overpartition to Hirschhorn and Sellers \cite{hirsch}. For a positive integer $s$, they defined $a_s(n)$ that counts the partitions of $n$ in which odd parts is assigned one of $s$ colors. The generating function of $a_s(n)$ is given by
	\begin{align}\label{e1a}
		\sum_{n=0}^{\infty}a_s(n)q^n=\dfrac{f_2^{s-1}}{f_1^s}.\end{align}
	The objective of this paper is to establish new infinite families of congruences for $\bar{a}_{s}(n)$. The following are our main results.
	\begin{theorem}\label{t2n1}
	Let $p\ge 3$ be a prime and let $m,\alpha,\beta \ge1$ be integers with $2\nmid \alpha$. If $\bar{a}_{2^m\alpha}(p)\equiv 0\pmod 2$, then for all $n,k\ge 0$ with $p\nmid(2n+1)$, we have
	\begin{align*}
	\bar{a}_{2^m\alpha}\left(2p^{2k+1}n+p^{2k+1}\right)\equiv 0\pmod {2^{m+2}}.
	\end{align*}
	If $\bar{a}_{2\beta+1}(p)\equiv 0\pmod 2$, then for $n,k\ge 0$ with $p\nmid(2n+1)$, we have
	\begin{align*}
		\bar{a}_{2\beta+1}\left(2p^{2k+1}n+p^{2k+1}\right)\equiv 0\pmod {4}.
	\end{align*}
	\end{theorem}
	\begin{theorem}\label{t4n1he}
	Let $k$ and $n$ be non-negative integers and let $m, \alpha,\beta\ge 1$ be integers with $2\nmid\alpha$. Let $p_1, p_2,\dots,p_{k+1}$ be primes such that $p_i\ge 3$ and $p_i\not\equiv1\pmod 4$ for each $1\le i\le k+1$. For any integer $j\not\equiv 0\pmod{p_{k+1}}$, we have
	\begin{align*}
	\bar{a}_{2^m\alpha}\left(4p_1^2p_2^2\cdots p_k^2p_{k+1}^2n+p_1^2p_2^2\cdots p_k^2p_{k+1}\left(j+p_{k+1}\right)\right)&\equiv 0\pmod {2^{m+2}}\\
	\bar{a}_{2\beta+1}\left(4p_1^2p_2^2\cdots p_k^2p_{k+1}^2n+p_1^2p_2^2\cdots p_k^2p_{k+1}\left(j+p_{k+1}\right)\right)&\equiv 0\pmod {4}.
	\end{align*}
	\end{theorem}
	The following corollary is obtained by setting $p_1=p_2=\dots=p_{k+1}=p$ in the above theorem.
	\begin{corollary}
	Let $k$ and $n$ be non-negative integers and let $m, \alpha,\beta\ge 1$ be integers with $2\nmid\alpha$. Let $p\ge 3$ be a prime with $p\equiv 3\pmod 4$. For any integer $j\not\equiv0\pmod p$, we have
	\begin{align*}
	\bar{a}_{2^m\alpha}\left(4p^{2k+2}n+p^{2k+1}j+p^{2k+2}\right)&\equiv 0\pmod {2^{m+2}}\\
	\bar{a}_{2\beta+1}\left(4p^{2k+2}n+p^{2k+1}j+p^{2k+2}\right)&\equiv 0\pmod {4}.
	\end{align*}
	\end{corollary}
	\begin{theorem}\label{t4n1t1}
	Let $k$ and $n$ be non-negative integers and let $m, \alpha,\beta\ge 1$ be integers with $2\nmid\alpha$. Let $p\ge 3$ be a prime such that $p\equiv 3\pmod 4$. Let $r$ be a non-negative integer. If $r$ is such that $p$ divides $4r+3$, then we have
	\begin{align*}
		\bar{a}_{2^m\alpha}(4p^{k+1}+4pr+3p)&\equiv p^2 \bar{a}_{2^m\alpha}\left(4p^{k-1}n+\frac{4r+3}{p}\right)\pmod{2^{m+2}}\\
		\bar{a}_{2\beta+1}(4p^{k+1}+4pr+3p)&\equiv p^2 \bar{a}_{2\beta+1}\left(4p^{k-1}n+\frac{4r+3}{p}\right)\pmod{4}.
	\end{align*}
	\end{theorem}
	\begin{corollary}\label{c4n12}
		Let $k$ be a non-negative integer and let $m, \alpha,\beta\ge 1$ be integers with $2\nmid\alpha$. Let $p\ge 3$ be a prime such that $p\equiv 3\pmod 4$. For all $n\ge 0$, we have
		\begin{align*}
			\bar{a}_{2^m\alpha}(4p^{2k}n+p^{2k})&\equiv (p^2)^k\bar{a}_{2^m\alpha}(4n+1)\pmod{2^{m+2}}\\
			\bar{a}_{2\beta+1}(4p^{2k}n+p^{2k})&\equiv (p^2)^k\bar{a}_{2\beta+1}(4n+1)\pmod{4}.
		\end{align*}
	\end{corollary}
	\begin{theorem}\label{t4n1}
	Let $p\ge 5$ be a prime with $p\equiv 3\pmod 4$ and let $m, \alpha,\beta\ge 1$ be integers with $2\nmid\alpha$. If $p\nmid n$, then for $n,k\ge 0$, we have
	\begin{align*}
	\bar{a}_{2^m\alpha}\left(4p^{2k+1}n+p^{2k+2}\right)&\equiv 0\pmod {2^{m+2}}\\
	\bar{a}_{2\beta+1}\left(4p^{2k+1}n+p^{2k+2}\right)&\equiv 0\pmod {4}.
	\end{align*}
	\end{theorem}
	\begin{theorem}\label{t4n2}
		Let $p\ge 3$ be a prime and let $m,\alpha,\beta \ge1$ be integers. If $\bar{a}_{2\alpha}(2p)\equiv 0\pmod 2$, then for all $n,k\ge 0$ with $p\nmid(2n+1)$, we have
		\begin{align*}
			\bar{a}_{2\alpha}\left(4p^{2k+1}n+2p^{2k+1}\right)\equiv 0\pmod {4}.
		\end{align*}
		If $\bar{a}_{2\beta+1}(2p)\equiv 0\pmod 2$, then for $n,k\ge 0$ with $p\nmid(2n+1)$, we have
		\begin{align*}
			\bar{a}_{2\beta+1}\left(4p^{2k+1}n+2p^{2k+1}\right)\equiv 0\pmod {8}.
		\end{align*}
	\end{theorem}
\begin{theorem}\label{t4n3}
 Let $p$ be a prime with $p\equiv 1\pmod 4$ and let $m, \alpha,\beta\ge 1$ be integers with $2\nmid\alpha$. If $\bar{a}_{2^m\alpha}(3p)\equiv 0\pmod 2$, then for all $n,k\ge 0$ with $p\nmid(4n+3)$, we have
 \begin{align*}
 \bar{a}_{2^m\alpha}\left(4p^{2k+1}n+3p^{2k+1}\right)\equiv 0\pmod {2^{m+3}}.
 \end{align*}
 If $\bar{a}_{2\beta+1}(3p)\equiv 0\pmod 2$, then for $n,k\ge 0$ with $p\nmid(4n+3)$, we have
 \begin{align*}
 	\bar{a}_{2\beta+1}\left(4p^{2k+1}n+3p^{2k+1}\right)\equiv 0\pmod {16}.
 \end{align*}
\end{theorem}
	\begin{theorem}\label{t8n4}
		Let $p\ge 3$ be a prime and let $m,\alpha,\beta \ge1$ be integers. If $\bar{a}_{2\alpha}(4p)\equiv 0\pmod 2$, then for $n,k\ge 0$ with $p\nmid(2n+1)$, we have
	\begin{align*}
		\bar{a}_{2\alpha}\left(8p^{2k+1}n+4p^{2k+1}\right)\equiv 0\pmod {8}.
	\end{align*}
	If $\bar{a}_{2\beta+1}(4p)\equiv 0\pmod 2$, then for $n,k\ge 0$ with $p\nmid(2n+1)$, we have
	\begin{align*}
		\bar{a}_{2\beta+1}\left(8p^{2k+1}n+4p^{2k+1}\right)\equiv 0\pmod {4}.
	\end{align*}
\end{theorem}
\begin{theorem}\label{t8n5}
		Let $m,\alpha,\beta \ge1$ be integers with $2\nmid \alpha$ and let $c(n)$ be defined by
	\begin{align}\label{e8n51}
		\sum_{n=0}^{\infty}c(n)q^n:= f_1^3 f_2^6,
	\end{align}
	and $p\ge 3$ be a prime. Define
	\begin{align}\label{e8n52}
		\kappa(p):= c\displaystyle\left(\frac{5(p^2-1)}{8}\right) + p^3\displaystyle\left(\frac{-\frac{5}{4}(p^2-1)}{p}\right)_L.
	\end{align}
	(1). For $n,k\ge0$, if $p\nmid n$, then
	\begin{align}\label{e8n53}
		\bar{a}_{2^m\alpha}\displaystyle\left(8p^{\nu(p)(k+1)-1}n + 5p^{\nu(p)(k+1)}\right) &\equiv0\pmod{2^{m+3}},\\
		\bar{a}_{2\beta+1}\left(8p^{\nu(p)(k+1)-1}n + 5p^{\nu(p)(k+1)}\right)&\equiv 0\pmod {16},
	\end{align}
	where
	\begin{align}\label{e8n54}
		\nu(p) :=\begin{cases}
			4, & \text{if } \kappa(p) \equiv 0\pmod 2,\\[4pt]
			6, & \text{if } \kappa(p) \equiv 1\pmod2.
		\end{cases}
	\end{align}
	(2). If $\kappa(p)\equiv1\pmod2$, then for $n,k\ge0$ with $p\nmid (8n+5)$,
	\begin{align}\label{e8n55}
		\bar{a}_{2^m\alpha}\displaystyle\left(8p^{6k+2}n + 5p^{6k+2}\right) &\equiv0\pmod {2^{m+3}}\\
		\bar{a}_{2\beta+1}\left(8p^{6k+2}n + 5p^{6k+2}\right)&\equiv 0\pmod {16}.
	\end{align}
\end{theorem}
 \begin{theorem}\label{t8n6}
 	Let $p$ be a prime with $p\equiv 1\pmod 4$ and let $s$ be a positive integer. If $\bar{a}_s(6p)\equiv 0\pmod 2$, then for $n,k\ge 0$ with $p\nmid(4n+3)$, we have
 	\begin{align*}
 		\bar{a}_s\left(8p^{4k+1}n+6p^{4k+1}\right)\equiv 0\pmod {16}.
 	\end{align*}
 \end{theorem}
	\begin{theorem}\label{t8n7}
	Let $m,\alpha,\beta \ge1$ be integers with $2\nmid \alpha$ and let $b(n)$ be defined by
	\begin{align}\label{e1.1}
		\sum_{n=0}^{\infty}c_4(n)q^n:= f_1 f_2^{10},
	\end{align}
	and $p\ge 3$ be a prime. Define
	\begin{align}\label{e1.2}
		\kappa(p):= c_4\displaystyle\left(\frac{7(p^2-1)}{8}\right) + p^4\displaystyle\left(\frac{-\frac{7}{4}(p^2-1)}{p}\right)_L.
	\end{align}
	(1). For $n,k\ge0$, if $p\nmid n$, then
	\begin{align}\label{e1.3}
		\bar{a}_{2^m\alpha}\displaystyle\left(8p^{\nu(p)(k+1)-1}n + 7p^{\nu(p)(k+1)}\right) &\equiv0\pmod{2^{m+4}},\\
		\bar{a}_{2\beta+1}\left(8p^{\nu(p)(k+1)-1}n + 7p^{\nu(p)(k+1)}\right)&\equiv 0\pmod {128},
	\end{align}
	where
	\begin{align}\label{e1.4}
		\nu(p) :=\begin{cases}
			4, & \text{if } \kappa(p) \equiv 0\pmod 2,\\[4pt]
			6, & \text{if } \kappa(p) \equiv 1\pmod2.
		\end{cases}
	\end{align}
	(2). If $\kappa(p)\equiv1\pmod2$, then for $n,k\ge0$ with $p\nmid (8n+7)$,
	\begin{align}\label{e1.5}
		\bar{a}_{2^m\alpha}\displaystyle\left(8p^{6k+2}n + 7p^{6k+2}\right) &\equiv0\pmod {2^{m+4}}\\
		\bar{a}_{2\beta+1}\left(8p^{6k+2}n + 7p^{6k+2}\right)&\equiv 0\pmod {128}.
	\end{align}
\end{theorem}

The structure of the paper is as follows. In Section \ref{s2}, we recall some definitions and basic facts on modular forms, and necessary results of Newman employed to prove our main theorems. In Section \ref{s3}, we obtain generating functions modulo powers of $2$ and provide the proof of new infinite families of congruences for $a_s(n)$ stated above. 
\section{Preliminaries}\label{s2}
	In this section, we collect necessary definitions, results and properties of modular forms. Throughout this paper, we consider only integer weight modular forms. For interest of readers, we cite the following book \cite{wom}.\\
\indent  For a positive integer $N$, $\Gamma_0(N)$ denotes a subgroup of $SL_2(\mathbb{Z})$ that is defined by
\begin{align*}
	\Gamma_0(N)&:= 
	\left\{
	\begin{bmatrix}
		a & b \\
		c & d
	\end{bmatrix} \in SL_2(\mathbb{Z})
	\,:\, c \equiv 0\!\!\!\! \pmod N 
	\right\}.
\end{align*}
We note that $\gamma=\begin{bmatrix}
	a & b \\
	c & d
\end{bmatrix}\in SL_2(\mathbb{Z})$ acts on $\mathbb{H}=\{z: Im(z)>0\}$, the upper half of complex plane $\mathbb{C}$, by bilinear transformation defined by $\gamma z:=\dfrac{az
	+b}{cz+d} $. If $\chi$ is a Dirichlet character modulo $N$ and $k$ is a positive integer, then any meromorphic function $f(z)$ on $\mathbb{H}$ which satisfies
	\begin{align*}
		f(\gamma z)=\chi(d)(cz+d)^kf(z),
	\end{align*} 
	for all $\gamma\in\Gamma_0(N)$ and $z\in\mathbb{H}$ is called a modular form of weight $k$ and Nebentypus character $\chi$ with respect to $\Gamma_0(N)$. If $f(z)$ is holomorphic on $\mathbb{H}$ and at all the cusps of $\Gamma_0(N)$, then $f(z)$ is a holomorphic modular form of weight $k$ with respect to $\Gamma_0(N)$ and character $\chi$. The space of such modular forms is denoted by $M_k(\Gamma_0(N), \chi)$.\\
\indent Recall the Dedekind's eta-function $\eta(z)$ defined by
\begin{align}\label{2.1}
	\eta(z):=q^{1/24}(q;q)_\infty=q^{1/24}\prod_{n=0}^\infty(1-q^{n}), \text{ where } q=e^{2\pi iz},
\end{align}
which is a non-vanishing holomorphic function on $\mathbb{H}$. A function is called an eta-quotient if it is of the form
\begin{align}
f(z)=\prod_{\delta\mid N}\eta(\delta z)^{r_\delta},
\end{align}
where $N$ is a positive integer and $r_{\delta}$ is an integer.\\ 
\indent If $f(z)$ is an eta-quotient associated with positive integer weight $k$ satisfying the conditions of the following result due to Gordon-Hughes \cite{gordon} and  Newman \cite{newmf}, and is holomorphic at all the cusps of $\Gamma_0(N)$, then $f(z) \in M_k(\Gamma_0(N), \chi )$. 
\begin{theorem}[{{\cite[Theorem~1.64]{wom}}}]\label{t2.1}
	If $ f(z)= \prod_{\delta \mid N} \eta(\delta z)^{r_\delta}$ is an eta-quotient with
	$k= \frac{1}{2} \sum_{\delta \mid N}{r_\delta} \in \mathbb{Z}$, and satisfies the following additional properties:
	\begin{align}\label{2.3}
		\sum_{\delta\mid N} \delta {r_\delta} \equiv 0 \pmod {24},
	\end{align}
	\begin{align}\label{2.4}
		\sum_{\delta \mid N} \frac{N}{\delta}  {r_\delta} \equiv 0 \pmod {24},
	\end{align}
	then $f(z)$ satisfies 
	\begin{align*}
		f\left(\dfrac{az+b}{cz+d}\right)=\chi(d)(cz+d)^kf(z)
	\end{align*} 
	for every $\begin{bmatrix}
		a & b \\
		c & d
	\end{bmatrix}\in \Gamma_0(N)$, where the character $\chi$
	is defined by
	$\chi (d) := \bigg( \frac{(-1)^k \prod_{\delta \mid N} \delta^{r_\delta}}{d} \bigg).$
\end{theorem}
To confirm the holomorphicity of $f(z)$ at cusps of $\Gamma_0(N)$, it is enough to verify that the orders at the cusps are non-negative. The next theorem of Ligozat \cite{ligozat} gives the sufficient condition to evaluate orders of an eta-quotient at cusps.
\begin{theorem}[{{\cite[Theorem~1.65]{wom}}}]\label{t2.2}
	Let $c,d,$ and $N$ be positive integers with $d\mid N$ and $gcd(c,d)=1$. If $f(z)$ is an eta-quotient satisfying the conditions of Theorem \ref{t2.1} for $N$, then the order of vanishing of $f(z)$ at the cusp $\frac{c}{d}$ is 
	\begin{align*}
		\dfrac{N}{24}\sum_{\delta\mid N} \frac{gcd(d,\delta)^2 r_\delta}{ gcd(d,\frac{N}{d})d\delta}.		
	\end{align*}
\end{theorem}
The following are the definitions of the Hecke operator and Hecke eigenform which play a vital role in our proof.
\begin{definition}[{{\cite[Definition~2.1]{wom}}}]
	Let $m$ be a positive integer and $f(z) = \sum_{n=0}^ \infty a(n)q^n \in  M_k(\Gamma_0(N), \chi ).$ The Hecke operator $T_m$ acts on $f(z)$ by 
	\begin{align}
		f(z)\mid {T_m} := \sum_{n=0}^\infty \bigg( \sum_{d\mid gcd(n,m)} \chi (d) d^{k-1}a \bigg(\frac{nm}{d^2} \bigg) \bigg)q^n.
	\end{align}
	In particular, if $m=p$ is a prime, then 
	\begin{align}\label{2.7}
		f(z)\mid {T_p} := \sum_{n=0}^\infty \bigg( a(pn)+ \chi (p) p^{k-1}a \bigg(\frac{n}{p} \bigg) \bigg)q^n.
	\end{align}
	We adopt the convention that $a(n/p)=0$ whenever $p\nmid n$.
\end{definition}
\begin{definition}[{{\cite[Definition~2.5]{wom}}}] 
A modular form $f(z)\in M_k(\Gamma_0(N), \chi)$ is called a Hecke eigenform if for every $m\ge 2$ there is a complex number $\lambda(m)$ for which
\begin{align}
f(z)\mid T_m=\lambda(m)f(z).
\end{align}
\end{definition}
	We now recall the definition of Legendre symbol required for the results that follow. Let $p$ be any odd prime and $n$ be any integer relatively prime to $p$. The Legendre symbol $\left(\dfrac{n}{p}\right)_L$ is defined by
\[
\left( \frac{n}{p} \right)_L :=
\begin{cases}
	\hphantom{-}1, & \text{if } n \text{ is a quadratic residue modulo }p\\[4pt]
	-1, & \text{if } n \text{ is a non-quadratic residue modulo } p.
\end{cases}
\]
Clearly, $\left(\frac{n}{p}\right)_L=0$ for $p|n$.
We next state the following results of Newman which provides an essential argument to prove our theorems.
\begin{lemma}[{{\cite[Theorem~1]{newman53}}}]\label{lnew53}
	Let $p_r(n)$ be defined by 
	\begin{align*}
		\sum_{n=0}^{\infty} p_r(n)q^n= f_1^r
	\end{align*}
	and $r$ be an even integer with $0<r\le 24$. If $p$ is a prime such that $r(p-1)\equiv 0\pmod {24}$ and $\delta =\dfrac{r(p-1)}{24}$, then the following identity holds
	\begin{align*}
		p_r(np+\delta)=p_r(\delta)p_r(n)-p^{(r/2)-1}p_r\left(\dfrac{n-\delta}{p}\right).
	\end{align*}
\end{lemma}
\begin{lemma}[{{\cite[Theorem~3]{newman55}}}]\label{lnew55}
Suppose that $r$ is one of the numbers $2,4,6,8,10, 14, 26$. Let $p\ge 5$ be a prime such that $r(p+1)\equiv0\pmod{24}$. Let $\Delta=r(p^2-1)/24$, and define $p_r(\alpha)=0$, if $\alpha$ is not a non-negative integer. Then 
\begin{align}
p_r(np+\Delta)=(-p)^{(r/2)-1}p_r(n/p).
\end{align}
\end{lemma}
\begin{lemma}[{{\cite[Theorem~3]{newman59}}}]\label{lnew59}
Let $p$ and $q$ be distinct primes. The values of non-zero integers $r$ and $s$ are those given in the table \cite{newman59}, where with an entry $(r,s)$ we must also include $(s,r)$, and $m$ is an integer such that $p$ satisfies $p\equiv 1\pmod m$. We set
	\begin{align*}
	\phi(\tau)&=\prod_{n=1}^{\infty}(1-x^n)^r(1-x^{nq})^s=\sum_{n=0}^{\infty}c(n)x^n,\\
	(\varepsilon, t)&= \displaystyle\left(\dfrac{(r+s)}{2}, \dfrac{(r+sq)}{24}\right),\\
	\delta&=t(p-1).
\end{align*}
For these values, the coefficients $c(n)$ of $\phi(\tau)$ satisfy 
\begin{align*}
c\left(np+\delta\right)=\beta c(n)-\gamma p^{\varepsilon-1}c\left(\dfrac{n-\delta}{p}\right),
\end{align*}
where
\begin{align*}
\gamma&= (-1)^{\varepsilon(p-1)/2}\left(\dfrac{q}{p}\right)^r
\end{align*}
and \[\beta=c(\delta)+\gamma p^{\varepsilon-1}c\left(\dfrac{-\delta}{p}\right)=\begin{cases}
c(\delta),& \delta>0\\[4pt]
1+\gamma p^{\varepsilon-1},& \delta=0.
\end{cases}\]

\end{lemma}
\begin{lemma}[{{\cite[Theorem~3]{newman62}}}]\label{lnew62}
	Let $p$ and $q$ be distinct primes. The values of integers $r$ and $s$ such that $rs\not=0$ and $r\not\equiv s\pmod2$ are those given in the table \cite{newman62}, where with the entry $(r,s)$ we must also include $(s,r)$. We set
	\begin{align*}
		\phi(\tau)&=\prod_{n=1}^{\infty}(1-x^n)^r(1-x^{nq})^s=\sum_{n=0}^{\infty}c(n)x^n,\\
		(\varepsilon, t)&= \displaystyle\left(\dfrac{(r+s)}{2}, \dfrac{(r+sq)}{24}\right),\\
		\Delta&=t(p^2-1).
	\end{align*}
	For these values, the coefficients $c(n)$ of $\phi(\tau)$ satisfy
	\begin{align*}
		c(np^2+\Delta)-\gamma c(n)+p^{2\varepsilon-2}c\displaystyle\left(\dfrac{n-\Delta}{p^2}\right)=0\;,
	\end{align*}
	where 
	\begin{align*}
		\gamma=p^{2\varepsilon-2}\alpha-\displaystyle\left(\frac{\theta}{p}\right)_L p^{\varepsilon-3/2}\displaystyle\left(\frac{n-\Delta}{p}\right)_L,
	\end{align*}
	$\theta=(-1)^{(1/2)-\varepsilon} 2q^s$, and
	$\alpha$ is a constant.
\end{lemma}
Throughout the sequel, the following lemma which can be easily proved using binomial theorem, will be frequently used without explicit reference.
\begin{lemma}\label{bt}
	For any prime $p$ and positive integers $k$ and $m$, we have
	\begin{align*}
		f_m^{p^k} \equiv f_{mp}^{p^{k-1}} \pmod{p^k}.
	\end{align*}
\end{lemma}
\section{Congruences modulo powers of $2$}\label{s3}
In the following lemma we list the dissections required to prove Lemma \ref{lgf1}-\ref{lgf3}.
\begin{lemma}\label{ld}
	The following $2$-dissections hold:
	\begin{align}
		\dfrac{1}{f_1^2}&=\dfrac{f_8^5}{f_2^5f_{16}^2}+2q\dfrac{f_4^2f_{16}^2}{f_2^5f_8} \label{edf2}\\
		\dfrac{1}{f_1^4}&=\dfrac{f_4^{14}}{f_2^{14}f_8^4}+4q\dfrac{f_4^2f_8^4}{f_2^{10}} \label{edf4}\\
		\dfrac{1}{f_1^8}&=\dfrac{f_4^{28}}{f_2^{28}f_8^8}+8q\dfrac{f_4^{16}}{f_2^{24}}+16q^2 \dfrac{f_4^4f_8^8}{f_2^{20}} \label{edf8}.	
	\end{align}
\end{lemma}
\begin{proof}
Equations $\eqref{edf2}$ and $\eqref{edf4}$ are same as {{\cite[Eq. 1.9.4]{poq}}} and {{\cite[Eq. 1.10.1]{poq}}}, respectively. Equation \eqref{edf8} can be obtained from \eqref{edf4}.
\end{proof}
In the lemmas that follow, we prove generating functions necessary for the proofs our main theorems.
\begin{lemma}\label{lgf1}
For integers $n\ge 0$, $k\ge 1$ and $\alpha\ge 1$ with $2\nmid \alpha$, we have
	\begin{align}
		\sum_{n=0}^{\infty}\bar{a}_{2^k\alpha}(2n+1)q^n&\equiv 2^{k+1}f_1^{12}\pmod{2^{k+2}}\label{e2n1}\\
			\sum_{n=0}^{\infty}\bar{a}_{2^k\alpha}(4n+1)q^n&\equiv 2^{k+1}f_1^{6}\pmod{2^{k+2}}\label{e4n1}\\
			\sum_{n=0}^{\infty}\bar{a}_{2^k\alpha}(4n+3)q^n&\equiv 2^{k+2}f_1^{18}\pmod{2^{k+3}}\label{e4n}\\
			\sum_{n=0}^{\infty}\bar{a}_{2^k\alpha}(8n+5)q^n&\equiv 2^{k+2}f_1^{15}\pmod{2^{k+3}}\label{e8n5}\\
			\sum_{n=0}^{\infty}\bar{a}_{2^k\alpha}(8n+7)q^n&\equiv 2^{k+3}f_1^{21}\pmod{2^{k+4}}\label{e8n}.
	\end{align}
\end{lemma}
	\begin{proof}
	Thanks to \eqref{gf}, we have
	\begin{align}
		\sum_{n=0}^{\infty}\bar{a}_{2^k\alpha}(n)q^n=\dfrac{f_2^{3\cdot2^k\alpha-2}}{f_1^{2^{k+1}\alpha}f_4^{2^k\alpha-1}}.
	\end{align}
	Employing \eqref{edf2} in the above equation, we have
	\begin{align*}
		\sum_{n=0}^{\infty}\bar{a}_{2^k\alpha}(n)q^n&=\dfrac{f_2^{3\cdot2^k\alpha-2}}{f_4^{2^k\alpha-1}}\left(\dfrac{f_8^5}{f_2^5f_{16}^2}+2q\dfrac{f_4^2f_{16}^2}{f_2^5f_8}\right)^{2^k\alpha}\\
		&=\dfrac{f_2^{3\cdot2^k\alpha-2}}{f_4^{2^k\alpha-1}}\left(\dfrac{f_8^5}{f_2^5f_{16}^2}\right)^{2^k\alpha}\sum_{i=0}^{2^k\alpha}\binom{2^k\alpha}{i}2^iq^i\left(\dfrac{f_4^{2i}f_{16}^{4i}}{f_8^{6i}}\right).
	\end{align*}
	Extracting the terms of the form $q^{2n+1}$ from both sides of the above equation, and then replacing $q^{2}$ with $q$ yields
	\begin{align*}
		\sum_{n=0}^{\infty}\bar{a}_{2^k\alpha}(2n+1)q^n&=\dfrac{f_4^{5\cdot2^k\alpha-6}}{f_1^{2^{k+1}\alpha+2}f_2^{2^k\alpha-3}f_8^{2^{k+1}\alpha-4}}\sum_{i=0}^{2^k\alpha}\binom{2^k\alpha}{2i+1}2^{2i+1}q^i\left(\dfrac{f_2^{4i}f_{8}^{8i}}{f_4^{12i}}\right).
	\end{align*}
	Again, using \eqref{edf2} in the above equation, we obtain
	\begin{align}\label{el0}
		\sum_{n=0}^{\infty}\bar{a}_{2^k\alpha}(2n+1)q^n&=\dfrac{f_4^{5\cdot2^k\alpha-6}f_8^{3\cdot2^k\alpha+9}}{f_2^{6\cdot2^{k}\alpha+2}f_{16}^{2^{k+1}\alpha+2}}\left(\sum_{i=0}^{2^k\alpha}\binom{2^k\alpha}{2i+1}2^{2i+1}q^i\left(\dfrac{f_2^{4i}f_{8}^{8i}}{f_4^{12i}}\right)\right)\nonumber\\
		&\cdot\left(\sum_{j=0}^{2^k\alpha+1}\binom{2^k\alpha+1}{j}2^jq^j\left(\dfrac{f_4^{2j}f_{16}^{4j}}{f_8^{6j}}\right)\right).
	\end{align}
	Since, for $i\ge1$ and $j\ge4$,
	\begin{align*}
		2^{i+1}\binom{2^k\alpha}{2i+1}\equiv 0\pmod{2^{k+2}} \text{ and } 2^j\binom{2^k\alpha+1}{j}\equiv 0\pmod{2^{k+2}},
	\end{align*}
we have
	\begin{align}\label{el1}
	\sum_{n=0}^{\infty}\bar{a}_{2^k\alpha}(2n+1)q^n&\equiv2^{k+1} \dfrac{f_4^{5\cdot2^k\alpha-6}f_8^{3\cdot2^k\alpha+9}}{f_2^{6\cdot2^{k}\alpha+2}f_{16}^{2^{k+1}\alpha+2}}\pmod {2^{k+2}}.
	\end{align}
	Applying Lemma \ref{bt} to the above congruence yields \eqref{e2n1}.\\
\indent	Extracting the terms of the form $q^{2n}$ from both sides of \eqref{el1}, and then replacing $q^{2}$ by $q$ gives
	\begin{align}\label{el2}
		\sum_{n=0}^{\infty}\bar{a}_{2^k\alpha}(4n+1)q^n&\equiv2^{k+1} \dfrac{f_2^{5\cdot2^k\alpha-6}f_4^{3\cdot2^k\alpha+9}}{f_1^{6\cdot2^{k}\alpha+2}f_{8}^{2^{k+1}\alpha+2}}\pmod {2^{k+2}}.
	\end{align}
	Using Lemma \ref{bt} in the above congruence gives \eqref{e4n1}.\\
	\indent	Since, for $i\ge2$ and $j\ge6$,
	\begin{align*}
		2^{i+1}\binom{2^k\alpha}{2i+1}\equiv 0\pmod{2^{k+3}} \text{ and } 2^j\binom{2^k\alpha+1}{j}\equiv 0\pmod{2^{k+3}},
	\end{align*}
	it follows from \eqref{el0} that
	\begin{align*}
		\sum_{n=0}^{\infty}\bar{a}_{2^k\alpha}(2n+1)q^n
		&\equiv \dfrac{f_4^{5\cdot2^k\alpha-6}f_8^{3\cdot2^k\alpha+9}}{f_2^{6\cdot2^{k}\alpha+2}f_{16}^{2^{k+1}\alpha+2}}\left(2^{k+1}\alpha+2^{k+2}(2^k\alpha+1)q\cdot \dfrac{f_4^2f_{16}^4}{f_8^6}\right)\pmod {2^{k+3}}.
	\end{align*}
	Extracting the terms of the form $q^{2n+1}$ from both sides of the above equation, and then replacing $q^{2}$ with $q$ yields
	\begin{align}\label{el3}
		\sum_{n=0}^{\infty}\bar{a}_{2^k\alpha}(4n+3)q^n&\equiv2^{k+2}\dfrac{f_2^{5\cdot2^k\alpha-4}f_4^{3\cdot2^k\alpha+3}}{f_1^{6\cdot2^{k}\alpha+2}f_{8}^{2^{k+1}\alpha-2}}\pmod{2^{k+3}}.
	\end{align}
	Employing Lemma \ref{bt} in the above congruence gives \eqref{e4n}.\\
	\indent Using \eqref{edf2} in \eqref{el2}, we obtain
\begin{align*}
	\sum_{n=0}^{\infty}\bar{a}_{2^k\alpha}(4n+1)q^n&\equiv2^{k+1} \dfrac{f_2^{5\cdot2^k\alpha-6}f_4^{3\cdot2^k\alpha+9}}{f_{8}^{2^{k+1}\alpha+2}}\left(\dfrac{f_8^5}{f_2^5f_{16}^2}\right)^{3\cdot2^k\alpha+1}\\
	&\cdot\left(\sum_{i=0}^{3\cdot2^k\alpha+1}\binom{3\cdot 2^k \alpha+1}{i}2^{i}q^{i}\dfrac{f_4^{2i}f_{16}^{4i}}{f_8^{6i}}\right)\pmod {2^{k+2}}.
\end{align*}
	Extracting the terms of the form $q^{2n+1}$ from both sides of the above equation, and then replacing $q^{2}$ with $q$ gives
\begin{align*}
\sum_{n=0}^{\infty}\bar{a}_{2^k\alpha}(8n+5)q^n\equiv 2^{k+2}\dfrac{f_2^{3\cdot2^k\alpha+11}f_4^{13\cdot 2^k\alpha -3}}{f_1^{10\cdot2^k\alpha+11}f_8^{6\cdot 2^k\alpha-2}}\pmod{2^{k+3}}
\end{align*}
Employing Lemma \ref{bt} in the above congruence results in \eqref{e8n5}.\\
\indent Using \eqref{edf2} in \eqref{el3}, we get
\begin{align*}
\sum_{n=0}^{\infty}\bar{a}_{2^k\alpha}(4n+3)q^n&\equiv2^{k+2}\dfrac{f_2^{5\cdot2^k\alpha-4}f_4^{3\cdot2^k\alpha+3}}{f_{8}^{2^{k+1}\alpha-2}}\left(\dfrac{f_8^5}{f_2^5f_{16}^2}\right)^{3\cdot2^k\alpha+1}\\
&\cdot\left(\sum_{i=0}^{3\cdot2^k\alpha+1}\binom{3\cdot2^k\alpha+1}{i}2^iq^i\dfrac{f_4^{2i}f_{16}^{4i}}{f_8^{6i}}\right)\pmod{2^{k+3}}.
\end{align*}
Extracting the terms of the form $q^{2n+1}$ from both sides of the above equation, and then replacing $q^{2}$ with $q$ yields
\begin{align*}
\sum_{n=0}^{\infty}\bar{a}_{2^k\alpha}(8n+7)q^n\equiv 2^{k+3}\dfrac{f_2^{3\cdot2^k\alpha+5}f_4^{13\cdot2^k\alpha+1}}{f_1^{10\cdot2^k\alpha+9}f_8^{6\cdot2^k\alpha-2}}\pmod{2^{k+4}}.	
\end{align*}
Employing Lemma \ref{bt} in the above congruence confirms \eqref{e8n}.
	\end{proof}
	\begin{lemma}\label{lgf2}
		For all $n\ge0$ and $\alpha\ge 1$, we have
	\begin{align}
		\sum_{n=0}^{\infty}\bar{a}_{2\alpha+1}(2n+1)q^n&\equiv 2f_1^{12}\pmod4\label{e2no}\\
		\sum_{n=0}^{\infty}\bar{a}_{2\alpha+1}(4n+1)q^n&\equiv 2f_1^{6}\pmod4\label{e4no}\\
		\sum_{n=0}^{\infty}\bar{a}_{2\alpha+1}(4n+2)q^n&\equiv 4f_1^{12}\pmod8\label{e4n2o}\\
		\sum_{n=0}^{\infty}\bar{a}_{2\alpha+1}(4n+3)q^n&\equiv 8f_1^{18}\pmod{16}\label{e4n3o}\\
		\sum_{n=0}^{\infty}\bar{a}_{2\alpha+1}(8n+4)q^n&\equiv 2f_1^{12}\pmod4\label{e8n4o}\\
		\sum_{n=0}^{\infty}\bar{a}_{2\alpha+1}(8n+5)q^n&\equiv 8f_1^{15}\pmod{16}\label{e8n5o}\\
	\sum_{n=0}^{\infty}\bar{a}_{2\alpha+1}(8n+7)q^n&\equiv 64f_1^{21}\pmod{128}\label{e8n7o}.
	\end{align}
\end{lemma}
	\begin{proof}
	Thanks to \eqref{gf}, we have
	\begin{align*}
	\sum_{n=0}^{\infty}\bar{a}_{2\alpha+1}(n)q^n=\dfrac{f_2^{6\alpha+1}}{f_1^{4\alpha+2}f_4^{2\alpha}}.
	\end{align*}
	Employing \eqref{edf2} and \eqref{edf4} in the above equation, we obtain
	\begin{align}
		\sum_{n=0}^{\infty}\bar{a}_{2\alpha+1}(n)q^n&=\dfrac{f_2^{6\alpha+1}}{f_4^{2\alpha}}\left(\dfrac{f_8^5}{f_2^5f_{16}^2}+2q\dfrac{f_4^2f_{16}^2}{f_2^5f_8}\right)\left(\dfrac{f_4^{14}}{f_2^{14}f_8^4}+4q\dfrac{f_4^2f_8^4}{f_2^{10}}\right)^\alpha\nonumber\\
		&=\dfrac{f_2^{6\alpha+1}}{f_4^{2\alpha}}\left(\dfrac{f_4^{14}}{f_2^{14}f_8^4}\right)^\alpha\left(\dfrac{f_8^5}{f_2^5f_{16}^2}+2q\dfrac{f_4^2f_{16}^2}{f_2^5f_8}\right)\left(\sum_{i=0}^{\infty}\binom{\alpha}{i}4^iq^i\dfrac{f_2^{4i}f_8^{8i}}{f_4^{12i}}\right)\nonumber
	\end{align}
	\begin{align}
		&=\dfrac{f_4^{12\alpha}}{f_2^{8\alpha-1}f_8^{4\alpha}}\left(\dfrac{f_8^5}{f_2^5f_{16}^2}\sum_{i=0}^{\alpha}\binom{\alpha}{i}4^iq^i\dfrac{f_2^{4i}f_8^{8i}}{f_4^{12i}}+2\dfrac{f_4^2f_{16}^2}{f_2^5f_8}\sum_{i=0}^{\alpha}\binom{\alpha}{i}4^iq^{i+1}\dfrac{f_2^{4i}f_8^{8i}}{f_4^{12i}}\right)\label{el8}
	\end{align}
		Extracting the terms of the form $q^{2n+1}$ from both sides of \eqref{el8}, and then replacing $q^{2}$ with $q$ gives
	\begin{align}
	\sum_{n=0}^{\infty}\bar{a}_{2\alpha+1}(2n+1)q^n&=\dfrac{f_2^{12\alpha}}{f_1^{8\alpha-1}f_4^{4\alpha}}\left(\dfrac{f_4^5}{f_1^5f_{8}^2}\sum_{i=0}^{\alpha}\binom{\alpha}{2i+1}4^{2i+1}q^i\dfrac{f_1^{8i+4}f_4^{16i+8}}{f_2^{24i+12}}\right.\nonumber\\
	&\left.+2\dfrac{f_2^2f_{8}^2}{f_1^5f_4}\sum_{i=0}^{\alpha}\binom{\alpha}{2i}4^{2i}q^{i}\dfrac{f_1^{8i}f_4^{16i}}{f_2^{24i}}\right)\nonumber\\
	&\equiv 2\dfrac{f_2^{12\alpha+2}f_8^2}{f_1^{8\alpha+4}f_4^{4\alpha+1}}\pmod 4.\label{el4}
	\end{align}
	Employing Lemma \ref{bt} in the above congruence yields \eqref{e2no}.\\
\indent	Using \eqref{edf4} and \eqref{edf8} in \eqref{el4}, we have
\begin{align}
\sum_{n=0}^{\infty}\bar{a}_{2\alpha+1}(2n+1)q^n&\equiv 2\dfrac{f_2^{12\alpha+2}f_8^2}{f_4^{4\alpha+1}}\left(\dfrac{f_4^{14}}{f_2^{14}f_8^4}+4q\dfrac{f_4^2f_8^4}{f_2^{10}}\right)\nonumber\\
&\cdot\left(\dfrac{f_4^{28}}{f_2^{28}f_8^8}+8q\dfrac{f_4^{16}}{f_2^{24}}+16q^2 \dfrac{f_4^4f_8^8}{f_2^{20}}\right)^\alpha\pmod 4\label{el6}\\
&\equiv 2 \dfrac{f_4^{24\alpha+13}}{f_2^{16\alpha+12}f_8^{8\alpha+2}}
\pmod4.\nonumber
\end{align}
	Extracting the terms of the form $q^{2n}$ from both sides of the above equation, and then replacing $q^{2}$ with $q$ yields
\begin{align}\label{el5}
\sum_{n=0}^{\infty}\bar{a}_{2\alpha+1}(4n+1)q^n\equiv 2\dfrac{f_2^{24\alpha+13}}{f_1^{16\alpha+12}f_4^{8\alpha+2}}\pmod 4.
\end{align}
	Employing Lemma \ref{bt} in the above congruence gives \eqref{e4no}.\\
	\indent From \eqref{el4}, we have
	\begin{align*}
	\sum_{n=0}^{\infty}\bar{a}_{2\alpha+1}(2n+1)q^n
	&\equiv 4\alpha\dfrac{f_2^{12\alpha-12}}{f_1^{8\alpha}f_4^{4\alpha-13}f_8^2}+2\dfrac{f_2^{12\alpha+2}f_8^2}{f_1^{8\alpha+4}f_4^{4\alpha+1}}\pmod {16}.
	\end{align*}
	Employing \eqref{edf4} and \eqref{edf8} in the above equation, we get
	\begin{align*}
	\sum_{n=0}^{\infty}\bar{a}_{2\alpha+1}(2n+1)q^n&\equiv 4\alpha\dfrac{f_2^{12\alpha-12}}{f_4^{4\alpha-13}f_8^2}\left(\dfrac{f_4^{28}}{f_2^{28}f_8^8}\right)^\alpha\left(\sum_{i=0}^{\alpha}\binom{\alpha}{i}8^iq^i\dfrac{f_2^{4i}f_8^{8i}}{f_4^{12i}}\right)+2\dfrac{f_2^{12\alpha+2}f_8^2}{f_4^{4\alpha+1}}\\
	&\left(\dfrac{f_4^{28}}{f_2^{28}f_8^{8}}\right)^\alpha\left(\dfrac{f_4^{14}}{f_2^{14}f_8^4}+4q\dfrac{f_4^2f_8^4}{f_2^{10}}\right)\left(\sum_{j=0}^{\alpha}\binom{\alpha}{j}8^jq^j\dfrac{f_2^{4j}f_8^{8j}}{f_4^{12j}}\right)\\
	&\equiv 4\alpha\dfrac{f_4^{24\alpha+13}}{f_2^{16\alpha+12}f_8^{8\alpha+2}}\sum_{i=0}^{\alpha}\binom{\alpha}{i}8^iq^i\dfrac{f_2^{4i}f_8^{8i}}{f_4^{12i}}+2\dfrac{f_4^{24\alpha-1}}{f_2^{16\alpha-2}f_8^{8\alpha-2}}\left(\dfrac{f_4^{14}}{f_2^{14}f_8^4}\right.\\
	&\left.\sum_{j=0}^{\alpha}\binom{\alpha}{j}8^jq^j\dfrac{f_2^{4j}f_8^{8j}}{f_4^{12j}}+4\dfrac{f_4^2f_8^4}{f_2^{10}}\sum_{j=0}^{\alpha}\binom{\alpha}{j}8^jq^{j+1}\dfrac{f_2^{4j}f_8^{8j}}{f_4^{12j}}\right)\pmod{16}.
	\end{align*}
		Extracting the terms of the form $q^{2n+1}$ from both sides of the above equation, and then replacing $q^{2}$ with $q$ yields
	\begin{align}\label{el7}
	\sum_{n=0}^{\infty}\bar{a}_{2\alpha+1}(4n+3)q^n&\equiv8\dfrac{f_2^{24\alpha+1}}{f_1^{16\alpha+8}f_4^{8\alpha-6}} \pmod{16}.
	\end{align}
	Employing Lemma \ref{bt} in the above equation gives \eqref{e4n3o}.\\
\indent From \eqref{el8}, we have
\begin{align*}
\sum_{n=0}^{\infty}\bar{a}_{2\alpha+1}(n)q^n\equiv\dfrac{f_4^{12\alpha}}{f_2^{8\alpha+4}f_8^{4\alpha-5}f_{16}^2}+2q\dfrac{f_4^{12\alpha+2}f_{16}^2}{f_2^{8\alpha+4}f_8^{4\alpha+1}}\pmod 4.
\end{align*}
Extracting the terms of the form $q^{2n}$ from both sides of the above equation, and then replacing $q^{2}$ with $q$ gives
\begin{align}\label{el9}
\sum_{n=0}^{\infty}\bar{a}_{2\alpha+1}(2n)q^n\equiv\dfrac{f_2^{12\alpha}}{f_1^{8\alpha+4}f_4^{4\alpha-5}f_{8}^2}\pmod 4.
\end{align}
Employing \eqref{edf4} and \eqref{edf8} in \eqref{el9}, we obtain
\begin{align}\label{el11}
\sum_{n=0}^{\infty}\bar{a}_{2\alpha+1}(2n)q^n\equiv\dfrac{f_4^{24\alpha+19}}{f_2^{16\alpha+14}f_8^{8\alpha+6}}+4q\dfrac{f_4^{24\alpha+7}}{f_2^{16\alpha+10}f_8^{8\alpha-2}}\pmod 8.
\end{align}
Extracting the terms of the form $q^{2n+1}$ from both sides of the above equation, and then replacing $q^{2}$ with $q$ yields
	\begin{align}\label{el10}
	\sum_{n=0}^{\infty}\bar{a}_{2\alpha+1}(4n+2)q^n\equiv4\dfrac{f_2^{24\alpha+7}}{f_1^{16\alpha+10}f_4^{8\alpha-2}}\pmod 8.
	\end{align}
	Employing Lemma \ref{bt} in \eqref{el10} gives \eqref{e4n2o}.\\
	\indent	Extracting the terms of the form $q^{2n}$ from both sides of \eqref{el11}, and then replacing $q^{2}$ with $q$ yields 
		\begin{align*}
		\sum_{n=0}^{\infty}\bar{a}_{2\alpha+1}(4n)q^n\equiv\dfrac{f_2^{24\alpha+19}}{f_1^{16\alpha+14}f_4^{8\alpha+6}}\pmod 4.
		\end{align*}
			Employing \eqref{edf2} in the above equation, we obtain
		\begin{align*}
			\sum_{n=0}^{\infty}\bar{a}_{2\alpha+1}(4n)q^n\equiv\dfrac{f_2^{24\alpha+19}}{f_4^{8\alpha+6}}\left(\dfrac{f_8^5}{f_2^5f_{16}^2}\right)^{8\alpha+7}\sum_{i=0}^{8\alpha+7}\binom{8\alpha+7}{i}2^iq^i\dfrac{f_4^{2i}f_{16}^{4i}}{f_8^{6i}}\pmod 4.
		\end{align*}
		Extracting the terms of the form $q^{2n+1}$ from both sides of the above equation, and then replacing $q^{2}$ with $q$ yields
			\begin{align*}
			\sum_{n=0}^{\infty}\bar{a}_{2\alpha+1}(8n+4)q^n\equiv 2\dfrac{f_4^{40\alpha+29}}{f_1^{16\alpha+16}f_2^{8\alpha+4}f_8^{16\alpha+10}}\pmod 4.
			\end{align*}
			Employing Lemma \ref{bt} in the above equation gives \eqref{e8n4o}.\\
		\indent	Employing \eqref{edf2} in \eqref{el5}, we have
			\begin{align*}
			\sum_{n=0}^{\infty}\bar{a}_{2\alpha+1}(4n+1)q^n&\equiv 2\dfrac{f_2^{24\alpha+13}}{f_4^{8\alpha+2}}\left(\dfrac{f_8^5}{f_2^5f_{16}^2}\right)^{8\alpha+6}\sum_{i=0}^{8\alpha+6}\binom{8\alpha+6}{i}2^iq^i\dfrac{f_4^{2i}f_{16}^{4i}}{f_8^{6i}}\pmod 4.
			\end{align*}
			Extracting the terms of the form $q^{2n+1}$ from both sides of the above equation, and then replacing $q^{2}$ with $q$ gives
			\begin{align*}
			\sum_{n=0}^{\infty}\bar{a}_{2\alpha+1}(8n+5)q^n\equiv 8 \dfrac{f_4^{40\alpha+24}}{f_1^{16\alpha+17}f_2^{8\alpha}f_8^{16\alpha+8}}\pmod{16}.
			\end{align*}
			Employing Lemma \ref{bt} in the above equation results in \eqref{e8n5o}.\\
			\indent Employing \eqref{edf2} in \eqref{el7}, we have
			\begin{align*}
			\sum_{n=0}^{\infty}\bar{a}_{2\alpha+1}(4n+3)q^n\equiv 8\dfrac{f_2^{24\alpha+1}}{f_4^{8\alpha-6}}\left(\dfrac{f_8^5}{f_2^5f_{16}^2}\right)^{8\alpha+4}\sum_{i=0}^{8\alpha+4}\binom{8\alpha+4}{i}2^iq^i\dfrac{f_4^{2i}f_{16}^{4i}}{f_8^{6i}}\pmod{16}.
			\end{align*}
			Extracting the terms of the form $q^{2n+1}$ from both sides of the above equation, and then replacing $q^{2}$ with $q$ yields
			\begin{align*}
			\sum_{n=0}^{\infty}\bar{a}_{2\alpha+1}(8n+7)q^n\equiv 64\dfrac{f_4^{40\alpha+14}}{f_1^{16\alpha+19}f_2^{8\alpha-8}f_8^{16\alpha+4}}\pmod{128}.
			\end{align*}
			Employing Lemma \ref{bt} in the above equation, we arrive at \eqref{e8n7o}.
	\end{proof}

	\begin{lemma}\label{lgf3}
	For all $n\ge0$ and $\alpha, s\ge 1$, we have
	\begin{align}
		\sum_{n=0}^{\infty}\bar{a}_{2\alpha}(4n+2)q^n&\equiv 2f_1^{12}\pmod 4\label{e4n2e}\\
		\sum_{n=0}^{\infty}\bar{a}_{2\alpha}(8n+4)q^n&\equiv 4f_1^{12}\pmod{8}\label{e8n4e}\\
		\sum_{n=0}^{\infty}\bar{a}_{s}(8n+6)q^n&\equiv 8f_1^{18}\pmod{16}\label{e8n6}.
	\end{align}
	\end{lemma}
\begin{proof}
	Thanks to \eqref{gf}, we have
	\begin{align*}
	\sum_{n=0}^{\infty}\bar{a}_{2\alpha}(n)q^n=\dfrac{f_2^{6\alpha-2}}{f_1^{4\alpha}f_4^{2\alpha-1}}.
	\end{align*}
	Employing \eqref{edf4} in the above equation, we get
	\begin{align*}
	\sum_{n=0}^{\infty}\bar{a}_{2\alpha}(n)q^n=\dfrac{f_2^{6\alpha-2}}{f_4^{2\alpha-1}}\left(\dfrac{f_4^{14}}{f_2^{14}f_8^4}\right)^\alpha \sum_{i=0}^{\alpha}\binom{\alpha}{i}4^iq^i\dfrac{f_2^{4i}f_8^{8i}}{f_4^{12i}}.
	\end{align*}
	Extracting the terms of the form $q^{2n}$ from both sides of the above equation, and then replacing $q^{2}$ with $q$ yields
		\begin{align*}
			\sum_{n=0}^{\infty}\bar{a}_{2\alpha}(2n)q^n&=\dfrac{f_1^{6\alpha-2}}{f_2^{2\alpha-1}}\left(\dfrac{f_2^{14}}{f_1^{14}f_4^4}\right)^\alpha 
			\sum_{i=0}^{\alpha}\binom{\alpha}{2i}4^{2i}q^i\dfrac{f_1^{8i}f_4^{16i}}{f_2^{24i}}\\
			&\equiv\dfrac{f_2^{12\alpha+1}}{f_1^{8\alpha+2}f_4^{4\alpha}}\pmod 4.
		\end{align*}
		Employing \eqref{edf2} and \eqref{edf8} in the above equation, we obtain
		\begin{align}\label{el13}
			\sum_{n=0}^{\infty}\bar{a}_{2\alpha}(2n)q^n\equiv \dfrac{f_4^{24\alpha}}{f_2^{16\alpha+4}f_8^{8\alpha-5}f_{16}^2}+2q\dfrac{f_4^{24\alpha+2}f_{16}^2}{f_2^{16\alpha+4}f_8^{8\alpha+1}}\pmod 4.
		\end{align}
		Extracting the terms of the form $q^{2n+1}$ from both sides of \eqref{el13}, and then replacing $q^{2}$ with $q$ yields
		\begin{align}\label{el14}
		\sum_{n=0}^{\infty}\bar{a}_{2\alpha}(4n+2)q^n\equiv 2\dfrac{f_2^{24\alpha+2}f_8^2}{f_1^{16\alpha+4}f_4^{8\alpha+1}}\pmod 4.
		\end{align}
		Employing Lemma \ref{bt} in the above equation gives \eqref{e4n2e}.\\
	\indent	Extracting the terms of the form $q^{2n}$ from both sides of \eqref{el13}, and then replacing $q^{2}$ with $q$ yields
		\begin{align*}
		\sum_{n=0}^{\infty}\bar{a}_{2\alpha}(4n)q^n\equiv \dfrac{f_2^{24\alpha}}{f_1^{16\alpha+4}f_4^{8\alpha-5}f_8^2}\pmod 4.
		\end{align*}
		Employing \eqref{edf2} in the above equation, we obtain
		\begin{align*}
		\sum_{n=0}^{\infty}\bar{a}_{2\alpha}(4n)q^n\equiv \dfrac{f_2^{24\alpha}}{f_4^{8\alpha-5}f_8^2}\left(\dfrac{f_8^5}{f_2^5f_{16}^2}\right)^{8\alpha+2}\sum_{i=0}^{8\alpha+2}\binom{8\alpha+2}{i}2^iq^i\dfrac{f_4^{2i}f_{16}^{4i}}{f_8^{6i}}\pmod 4.
		\end{align*}
		Extracting the terms of the form $q^{2n+1}$ from both sides of \eqref{el13}, and then replacing $q^{2}$ with $q$ gives
		\begin{align*}
		\sum_{n=0}^{\infty}\bar{a}_{2\alpha}(8n+4)q^n&\equiv\dfrac{f_1^{24\alpha}}{f_2^{8\alpha-5}f_4^2}\left(\dfrac{f_4^5}{f_1^5f_8^2}\right)^{8\alpha+2}\sum_{i=0}^{8\alpha+2}\binom{8\alpha+2}{2i+1}2^{2i+1}q^i\dfrac{f_4^{4i+2}f_{16}^{8i+4}}{f_8^{12i+6}}\pmod 4\\
		&\equiv \dfrac{f_4^{40\alpha+2}}{f_1^{16\alpha+10}f_2^{8\alpha-7}f_8^{16\alpha}}\pmod 8.
		\end{align*}
		Employing Lemma \ref{bt} in the above equation results in \eqref{e8n4e}.\\
		\indent Employing \eqref{edf2} in \eqref{el14}, we obtain
		\begin{align*}
		\sum_{n=0}^{\infty}\bar{a}_{2\alpha}(4n+2)q^n\equiv 2\dfrac{f_2^{24\alpha+2}f_8^2}{f_4^{8\alpha+1}}\left(\dfrac{f_8^5}{f_2^5f_{16}^2}\right)^{8\alpha+2}\sum_{i=0}^{8\alpha+2}\binom{8\alpha+2}{i}2^iq^i\dfrac{f_4^{2i}f_{16}^{4i}}{f_8^{6i}}\pmod 4.
		\end{align*}
		Extracting the terms of the form $q^{2n+1}$ from both sides of the above equation, and then replacing $q^{2}$ with $q$ yields
		\begin{align*}
		\sum_{n=0}^{\infty}\bar{a}_{2\alpha}(8n+6)q^n\equiv 8\dfrac{f_4^{40\alpha+6}}{f_1^{16\alpha+8}f_2^{8\alpha-1}f_8^{16\alpha}}\pmod{16}.
		\end{align*}
			Employing Lemma \ref{bt} in the above equation, we get
		\begin{align}\label{el15}
			\sum_{n=0}^{\infty}\bar{a}_{2\alpha}(8n+6)q^n\equiv 8f_1^{18}\pmod{16}.
		\end{align}
		Employing \eqref{edf2} in \eqref{el10}, we obtain
		\begin{align*}
	    \sum_{n=0}^{\infty}\bar{a}_{2\alpha+1}(4n+2)q^n\equiv 4\dfrac{f_2^{24\alpha+7}}{f_4^{8\alpha-2}}\left(\dfrac{f_8^5}{f_2^5f_{16}^2}\right)^{8\alpha+5}\sum_{i=0}^{8\alpha+5}\binom{8\alpha+5}{i}2^iq^i\dfrac{f_4^{2i}f_{16}^{4i}}{f_8^{6i}}\pmod 8.
		\end{align*}
		Extracting the terms of the form $q^{2n+1}$ from both sides of the above equation, and then replacing $q^{2}$ with $q$ gives
		\begin{align*}
		\sum_{n=0}^{\infty}\bar{a}_{2\alpha+1}(8n+6)q^n\equiv8\dfrac{f_4^{40\alpha+19}}{f_1^{16\alpha+18}f_2^{8\alpha-4}f_8^{16\alpha+6}}\pmod {16}.
		\end{align*}
		Employing Lemma \ref{bt} in the above equation, we get
	\begin{align}\label{el16}
    \sum_{n=0}^{\infty}\bar{a}_{2\alpha+1}(8n+6)q^n\equiv 8f_1^{18}\pmod{16}
	\end{align}
	Combining \eqref{el15} and \eqref{el16}, we arrive at \eqref{e8n6}.
\end{proof}
\begin{proof}[Proof of Theorem \ref{t2n1}\unskip]
Define 
\begin{align}\label{e1.3.1}
\sum_{n=0}^{\infty}c_1(n)q^n:=f_1^4f_2^4.
\end{align}
Thanks to Lemma \ref{lnew59}, for prime $p\equiv1 \pmod 2$ and all $n\ge 0$, we have
\begin{align*}
c_1\left(np+\dfrac{p-1}{2}\right)=c_1\left(\dfrac{p-1}{2}\right)c_1(n)-(-1)^{2(p-1)}	\left(\dfrac{2}{p}\right)^4 p^3 c_1\left(\dfrac{n-\frac{p-1}{2}}{p}\right),
\end{align*}
which implies
\begin{align}\label{e1.3.2}
c_1\left(np+\dfrac{p-1}{2}\right)=c_1\left(\dfrac{p-1}{2}\right)c_1(n)-p^3c_1\left(\dfrac{n-\frac{p-1}{2}}{p}\right).
\end{align}
If $p\nmid(2n+1)$, then $c_1\left(\dfrac{n-\frac{p-1}{2}}{p}\right)=0$, since $\dfrac{n-\frac{p-1}{2}}{p}$ cannot be an integer. Hence, from \eqref{e1.3.2} we see that if $p\nmid(2n+1)$, then
\begin{align}\label{e1.3.3}
c_1\left(np+\dfrac{p-1}{2}\right)=c_1\left(\dfrac{p-1}{2}\right)c_1(n).
\end{align}
It follows from \eqref{e1.3.3} that if $p\nmid(2n+1)$ and $c_1\left(\dfrac{p-1}{2}\right)\equiv 0\pmod 2$, then 
\begin{align}\label{e1.3.4}
c_1\left(pn+\dfrac{p-1}{2}\right)\equiv 0\pmod 2.
\end{align}
Replacing $n$ by $pn+\dfrac{p-1}{2}$ in \eqref{e1.3.2}, we obtain
\begin{align}\label{e1.3.5}
c_1\left(p^2n+\dfrac{p^2-1}{2}\right)= c_1\left(\dfrac{p-1}{2}\right)c_1\left(pn+\dfrac{p-1}{2}\right)-p^3c_1(n).
\end{align}
From \eqref{e1.3.5}, it is clear that if $c_1\left(\dfrac{p-1}{2}\right)\equiv 0\pmod 2$, then for $n\ge0$, 
\begin{align}\label{e1.3.6}
c_1\left(p^2n+\dfrac{p^2-1}{2}\right)\equiv c_1(n)\pmod 2.
\end{align}
By \eqref{e1.3.6} and mathematical induction, if $c_1\left(\dfrac{p-1}{2}\right)\equiv 0\pmod 2$, then for $n\ge 0$ and $k\ge 0$, we have
\begin{align}\label{e1.3.7}
c_1\left(p^{2k}n+\dfrac{p^{2k}-1}{2}\right)\equiv c_1(n)\pmod 2.
\end{align}
Replacing $n$ by $pn+\dfrac{p-1}{2}$ in \eqref{e1.3.7} and using \eqref{e1.3.4}, we deduce that if $p\nmid(2n+1)$ and $c_1\left(\dfrac{p-1}{2}\right)\equiv 0\pmod 2$, then 
\begin{align}\label{e1.3.8}
c_1\left(p^{2k+1}n+\dfrac{p^{2k+1}-1}{2}\right)\equiv 0\pmod 2.
\end{align}
Since $\bar{a}_s(n)\equiv 0\pmod2$ for all $n\ge 1$, though not required, we illustrate and include the proof of the case $c_1\left(\dfrac{p-1}{2}\right)\equiv 1\pmod 2$ for completeness.\\
From \eqref{e1.3.2}, we obtain that if $c_1\left(\dfrac{p-1}{2}\right)\equiv 1\pmod 2$, then 
\begin{align}\label{e1.3.9}
c_1\left(pn+\dfrac{p-1}{2}\right)\equiv c_1(n)+c_1\left(\dfrac{n-\frac{p-1}{2}}{p}\right)\pmod 2.
\end{align}
Replacing $n$ by $pn+\dfrac{p-1}{2}$ in \eqref{e1.3.9} yields
\begin{align*}
c_1\left(p^2n+\dfrac{p^2-1}{2}\right)\equiv c_1\left(pn+\dfrac{p-1}{2}\right)+c_1(n)\pmod 2.
\end{align*}
Hence, from above equation and \eqref{e1.3.9}, we deduce that
\begin{align}\label{e1.3.10}
c_1\left(p^2n+\dfrac{p^2-1}{2}\right)\equiv c_1\left(\dfrac{n-\frac{p-1}{2}}{p}\right)\pmod 2.
\end{align}
By \eqref{e1.3.10}, we find that if $c_1\left(\dfrac{p-1}{2}\right)\equiv 1\pmod 2$ and $p\nmid(2n+1)$, then 
\begin{align}\label{e1.3.11}
c_1\left(p^2n+\dfrac{p^2-1}{2}\right)\equiv 0\pmod 2.
\end{align}
Replacing $n$ by $pn+\dfrac{p-1}{2}$ in \eqref{e1.3.10} yields
\begin{align}\label{e1.3.12}
c_1\left(p^3n+\dfrac{p^3-1}{2}\right)\equiv c_1(n)\pmod 2.
\end{align}
By \eqref{e1.3.12} and mathematical induction, for $n,k\ge 0$, we have
\begin{align}\label{e1.3.13}
c_1\left(p^{3k}n+\dfrac{p^{3k}-1}{2}\right)\equiv c_1(n)\pmod 2.
\end{align}
Replacing $n$ by $p^2n+\dfrac{p^2-1}{2}$ in \eqref{e1.3.13} and using \eqref{e1.3.11}, we see that if $c_1\left(\dfrac{p-1}{2}\right)\equiv 1\pmod 2$, then for $n,k\ge 0$ with $p\nmid (2n+1)$, we have
\begin{align}\label{e1.3.14}
c_1\left(p^{3k+2}n+\dfrac{p^{3k+2}-1}{2}\right)\equiv 0\pmod 2.
\end{align}
From \eqref{e2n1}, \eqref{e2no}, and \eqref{e1.3.1} for $n\ge 0$, $m,\alpha,\beta\ge 1$ with $2\nmid\alpha$, we have
\begin{align}
\bar{a}_{2^m\alpha}(2n+1)&\equiv 2^{m+1}c_1(n)\pmod{2^{m+2}}\label{e1.3.15}\\
\bar{a}_{2\beta+1}(2n+1)&\equiv 2c_1(n)\pmod 4.\label{e1.3.16}
\end{align}
Therefore, the desired result follows from \eqref{e1.3.8}, \eqref{e1.3.15} and \eqref{e1.3.16}.
\end{proof}
\begin{proof}[Proof of Theorem \ref{t4n1he}\unskip]
Thanks to \eqref{e4n1}, we have
\begin{align*}
\sum_{n=0}^{\infty}\bar{a}_{2^m\alpha}(4n+1)q^{4n+1}\equiv 2^{m+1}q f_4^6\pmod{2^{m+2}},
\end{align*}
which implies
\begin{align}\label{e1.4.1}
\sum_{n=0}^{\infty}\bar{a}_{2^m\alpha}(4n+1)q^{4n+1}\equiv 2^{m+1}\eta^6(4z)\pmod{2^{m+2}}.
\end{align}
By Theorems \ref{t2.1} and \ref{t2.2}, it is easy to see that $\eta^6(4z)\in S_3\left(\Gamma_0(16), \left(\frac{-4^6}{.}\right)\right)$. Let $\sum_{n=1}^{\infty}b(n)$ be the Fourier expansion of  $\eta^6(4z)$. Then $b(n)=0$, if $n\not\equiv 1\pmod 4$ for all $n\ge 0$.\\ Thus, comparing coefficients in \eqref{e1.4.1}, we obtain
\begin{align}\label{e1.4.3}
\bar{a}_{2^m\alpha}(4n+1)\equiv 2^{m+1}b(n+1)\pmod{2^{m+2}}.
\end{align}
Since  $\eta^6(4z)$ is a Hecke eigenform (see \cite{martin}), we have
\begin{align}\label{e1.4.4}
\eta^6(4z)\mid T_p=\sum_{n=1}^{\infty}\left(b(pn)+p^2\left(\frac{-4^6}{p}\right)b\left(\frac{n}{p}\right)\right)q^n=\lambda(p)\sum_{n=1}^{\infty}b(n)q^n.
\end{align}
Comparing the coefficients in \eqref{e1.4.4}, we get
\begin{align}\label{e1.4.5}
b(pn)+p^2\left(\frac{-1}{p}\right)b\left(\frac{n}{p}\right)=\lambda(p)b(n).
\end{align}
 Setting $n=1$ in \eqref{e1.4.5}, and noting $b(1)=1$, we obtain
\begin{align}\label{e1.4.6}
b(p)=\lambda(p).
\end{align}
Since $b(p)=0$ for all $p\not\equiv 1\pmod 4$, equation \eqref{e1.4.6} implies
\begin{align*}
\lambda(p)=0, \text{ for all } p\not\equiv 1\pmod 4.	
\end{align*}
Also, note that $\left(\frac{-1}{p}\right)=-1$ if $p\not\equiv 1\pmod 4$. Thus, from \eqref{e1.4.5} we derive for all $p\not\equiv 1\pmod 4$,
\begin{align}\label{e1.4.7}
b(pn)-p^2b\left(\frac{n}{p}\right)=0.
\end{align}
Replacing $n$ by $pn+r$ with $gcd(n,r)=1$ in \eqref{e1.4.7}, we see that for all $n\ge 0$ with $p\nmid n$,
\begin{align}\label{e1.4.8}
b(p^2n+pr)=0.
\end{align}
Replacing $n$ by $4n-pr+1$ in \eqref{e1.4.8} and employing the resulting equation in \eqref{e1.4.3}, we obtain
\begin{align}\label{e1.4.10}
\bar{a}_{2^m\alpha}\left(4np^2+pr(1-p^2)+p^2\right)\equiv 0\pmod {2^{m+2}},
\end{align}
with $gcd(r,p)=1$.\\
Since $p\equiv 3\pmod 4$, we have $4\mid(1-p^2)$ and $gcd(1-p^2,p)=1$. Thus, when $r$ runs over a residue system excluding the multiples of $p$, so does $r(1-p^2)$. Therefore, for $p\nmid j$, \eqref{e1.4.10} can be rewritten as 
\begin{align}\label{e1.4.11}
\bar{a}_{2^m\alpha}(4p^2n+pj+p^2)\equiv 0\pmod {2^{m+2}}.
\end{align}
Replacing $n$ by $pn$ in \eqref{e1.4.7}, we obtain
\begin{align}\label{e1.4.12}
b(p^2n)=p^2b(n).
\end{align}
Further, replacing $n$ by $4n+1$ in \eqref{e1.4.12}, we get
\begin{align}\label{e1.4.13}
b(4p^2n+p^2)=p^2b(4n+1).
\end{align}
From \eqref{e1.4.3} and \eqref{e1.4.13}, we get
\begin{align}
\bar{a}_{2^m\alpha}(4p^2n+p^2)\equiv p^2\bar{a}_{2^m\alpha}(4n+1)\pmod{2^{m+2}}.
\end{align}
Let $p_i\ge 3$ be primes such that $p_i\equiv 3\pmod 4$. Since,
\begin{align*}
4p_1^2p_2^2\cdots p_k^2 n+p_1^2p_2^2\cdots p_k^2=4p_1^2\left(p_2^2\cdots p_k^2 n+\dfrac{p_2^2\cdots p_k^2 -1}{4}\right)+p_1^2,
\end{align*}
using \eqref{e1.4.13} repeatedly and then employing \eqref{e1.4.11}, we obtain
\begin{align*}
&\bar{a}_{2^m\alpha}\left(4p_1^2p_2^2\cdots p_k^2p_{k+1}^2n+p_1^2p_2^2\cdots p_k^2p_{k+1}\left(j+p_{k+1}\right)\right)\\
&\equiv p_1^2 \bar{a}_{2^m\alpha}\left(4p_1^2p_2^2\cdots p_k^2p_{k+1}^2n+p_1^2p_2^2\cdots p_k^2p_{k+1}\left(j+p_{k+1}\right)\right)\\
\vdots\\
&\equiv (p_1p_2\cdots p_k)^2\bar{a}_{2^m\alpha}\left(4p_{k+1}^2n+p_{k+1}(j+p_{k+1})\right)\\
&\equiv 0\pmod{2^{m+2}},
\end{align*}
when $j\not\equiv 0\pmod{p_{k+1}}$. This completes our proof.
\end{proof}
\begin{proof}[Proof of Theorem \ref{t4n1t1}\unskip]
		Replacing $n$ by $4n+3$, in \eqref{e1.4.7}, we see that for any prime $p\equiv 3\pmod 4$,
	\begin{align}\label{e1.4.14}
		b(4pn+3p)=p^2b\left(\dfrac{4n+3}{p}\right).
	\end{align}
	Replacing $n$ by $p^kn+r$ with $p\nmid r$ in \eqref{e1.4.14}, we get
	\begin{align}\label{e1.4.15}
	b\left(4\left(p^{k+1}n+pr+\dfrac{3p-1}{4}\right)+1\right)=p^2b\left(4\left(p^{k-1}n+\dfrac{4r+3-p}{4p}\right)+1\right).
	\end{align}
	Note that $\frac{3p-1}{4}$ and $\frac{4r+3-p}{4p}$ are integers. Combining \eqref{e1.4.3} and \eqref{e1.4.15}, we obtain the desired result.
\end{proof}
\begin{proof}[Proof of Corollary \ref{c4n12}\unskip]
Let $p$ be a prime such that $p\equiv 3\pmod 4$. 	We choose a non-negative integer $r$ such that $4r+3=p^{2k-1}$. Replacing $k$ by $2k-1$ in Theorem \ref{t4n1t1}, we obtain
\begin{align*}
\bar{a}_{2^m\alpha}\left(4p^{2k}n+p^{2k}\right)&\equiv p^2\bar{a}_{2^m\alpha}(4p^{2k-2}n+p^{2k-2})\\
&\equiv\hdots\equiv (-p^2)^{k}\bar{a}_{2^m\alpha}(4n+1)\pmod{2^{m+2}}.
\end{align*}
Hence the result.
\end{proof}
\begin{proof}[Proof of Theorem \ref{t4n1}\unskip]
	Define
	\begin{align}\label{e1.2.0}
		\sum_{n=0}^{\infty}c_2(n)q^n:=f_1^6
	\end{align}
	Thanks to Lemma \ref{lnew55}, for prime $p\equiv 3\pmod 4$ and all $n\ge 0$, we have
	\begin{align*}
		c_2\left(np+\dfrac{p^2-1}{4}\right)=p^2c_2\left(\dfrac{n}{p}\right),
	\end{align*}
	which implies
	\begin{align}\label{e1.2.1}
		c_2\left(np+\dfrac{p^2-1}{4}\right)\equiv c_2\left(\dfrac{n}{p}\right)\pmod 2.
	\end{align}
	If $p\nmid n$, then it follows from \eqref{e1.2.1} that
	\begin{align}\label{e1.2.2}
		c_2\left(np+\dfrac{p^2-1}{4}\right)\equiv 0\pmod 2.
	\end{align}
	Replacing $n$ by $pn$ in \eqref{e1.2.1}, we obtain
	\begin{align}\label{e1.2.3}
		c_2\left(p^2n+\dfrac{p^2-1}{4}\right)\equiv c_2(n)\pmod 2.
	\end{align}
	By mathematical induction and \eqref{e1.2.3}, for $n\ge 0$ and $k\ge0$, we obtain
	\begin{align}\label{e1.2.4}
		c_2\left(p^{2k}n+\dfrac{p^{2k}-1}{4}\right)\equiv c_2(n)\pmod 2.
	\end{align}
	Replacing $n$ by $pn+\dfrac{p^2-1}{4}$ in the above equation, we deduce that if $p\nmid n$, then for $n\ge 0$ and $k\ge 0$, we have
	\begin{align}\label{e1.2.5}
		c_2\left(np^{2k+1}+\dfrac{p^{2k+2}-1}{4}\right)\equiv 0\pmod 2.
	\end{align}
	From \eqref{e4n1}, \eqref{e4no}, and  \eqref{e1.2.0}, for $n\ge0$ and $m,\alpha,\beta\ge 1$, we have
	\begin{align}
		\bar{a}_{2^m\alpha}(4n+1)&\equiv 2^{m+1}c_2(n)\pmod{2^{m+2}}\label{e1.2.6}\\
		\bar{a}_{2\beta+1}(4n+1)&\equiv 2c_2(n)\pmod 4.\label{e1.2.7}
	\end{align}
	Therefore, the desired result follows from \eqref{e1.2.5}, \eqref{e1.2.6}, and \eqref{e1.2.7}.
\end{proof}
\begin{proof}[Proof of Theorem \ref{t4n2}\unskip]
	From \eqref{e4n2o}, \eqref{e4n2e}, and \eqref{e1.3.1} for $n\ge 0$, $m,\alpha,\beta\ge 1$, we have
	\begin{align}
		\bar{a}_{2\alpha}(4n+2)&\equiv 2c_1(n)\pmod{4}\label{e1.3.17}\\
		\bar{a}_{2\beta+1}(4n+2)&\equiv 4c_1(n)\pmod 8.\label{e1.3.18}
	\end{align}
	Therefore, the desired result follows from \eqref{e1.3.8}, \eqref{e1.3.17} and \eqref{e1.3.18}.
\end{proof}
\begin{proof}[Proof of Theorem \ref{t4n3}\unskip]
		Define
	\begin{align}
		\sum_{n=0}^{\infty}c_3(n)q^n:=f_1^{18}.\label{e1.1.0}
	\end{align}
	Thanks to Lemma \ref{lnew53}, for $p\equiv 1\pmod 4$ and all $n\ge 0$, we have
	\begin{align*}
		c_3\left(pn+\dfrac{3(p-1)}{4}\right)=c_3\left(\dfrac{3(p-1)}{4}\right)c_3(n)-p^8c_3\left(\frac{n-\frac{3(p-1)}{4}}{p}\right).
	\end{align*}
	Reducing the above to modulo $2$, we obtain
	\begin{align}
		c_3\left(pn+\dfrac{3(p-1)}{4}\right)\equiv c_3\left(\dfrac{3(p-1)}{4}\right)c_3(n)+c_3\left(\frac{n-\frac{3(p-1)}{4}}{p}\right)\pmod 2.\label{e1.1.1}
	\end{align}
	If $c_3\left(\dfrac{3(p-1)}{4}\right)\equiv 0\pmod 2$, it follows from \eqref{e1.1.1} that
	\begin{align}
		c_3\left(pn+\dfrac{3(p-1)}{4}\right)\equiv c_3\left(\frac{n-\frac{3(p-1)}{4}}{p}\right)\pmod 2.\label{e1.1.2}
	\end{align}
Hence, if $c_3\left(\dfrac{3(p-1)}{4}\right)\equiv 0\pmod 2$ and $p\nmid(4n+3)$, then from \eqref{e1.1.2} we arrive at
\begin{align}
	c_3\left(pn+\dfrac{3(p-1)}{4}\right)\equiv 0\pmod 2.	\label{e1.1.3}
	\end{align}
	Replacing $n$ by $pn+\frac{3(p-1)}{4}$ in \eqref{e1.1.1}, we get
	\begin{align}
		c_3\left(p^2n+\dfrac{3(p^2-1)}{4}\right)\equiv c_3\left(\dfrac{3(p-1)}{4}\right)c_3\left(pn+\dfrac{3(p-1)}{4}\right)+c_3(n)\pmod 2.\label{e1.1.4}
	\end{align}
	If $c_3\left(\dfrac{3(p-1)}{4}\right)\equiv 0\pmod 2$, then from \eqref{e1.1.4}, we see that for $n\ge 0$,
	\begin{align}
		c_3\left(p^2n+\dfrac{3(p^2-1)}{4}\right)\equiv c_3(n)\pmod 2.\label{e1.1.5}
	\end{align}
	In view of \eqref{e1.1.5} and mathematical induction, we arrive at
	\begin{align}
		c_3\left(p^{2k}n+\dfrac{3(p^{2k}-1)}{4}\right)\equiv c_3(n)\pmod 2.\label{e1.1.6}
	\end{align}
	On replacing $n$ by $pn+\frac{3(p-1)}{4}$ in \eqref{e1.1.6}and employing \eqref{e1.1.3}, we see that if $c_3\left(\frac{3(p-1)}{4}\right)\equiv 0\pmod 2$ and $p\nmid(4n+3)$, then for $n\ge 0$ and $k\ge 0$, we obtain
	\begin{align}
		c_3\left(p^{2k+1}n+\dfrac{3(p^{2k+1}-1)}{4}\right)\equiv 0\pmod 2.\label{e1.1.8}
	\end{align}
	Since $\bar{a}_s(n)\equiv 0\pmod2$ for all $n\ge 1$, though not required, we illustrate and include the proof of the case $c_3\left(\dfrac{p-1}{2}\right)\equiv 1\pmod 2$ for completeness.\\
	If $c_3\left(\frac{3(p-1)}{4}\right)\equiv 1\pmod 2$, then from \eqref{e1.1.1}, we have
	\begin{align}\label{e1.1.9a}
	c_3\left(pn+\dfrac{3(p-1)}{4}\right)\equiv c_3(n)+c_3\left(\dfrac{n-\frac{3(p-1)}{4}}{p}\right)\pmod 2.
	\end{align}
	Replacing $n$ by $pn+\dfrac{3(p-1)}{4}$ in \eqref{e1.1.9a}, we get
	\begin{align}
		c_3\left(p^2n+\dfrac{3(p^2-1)}{4}\right)\equiv c_3\left(pn+\dfrac{3(p-1)}{4}\right)+c_3(n)\pmod 2.\label{e1.1.9}
	\end{align}
	Using \eqref{e1.1.9a} in the above equation, we obtain
	\begin{align}\label{e1.1.9b}
	c_3\left(p^2n+\dfrac{3(p^2-1)}{4}\right)\equiv c_3\left(\dfrac{n-\frac{3(p-1)}{4}}{p}\right)\pmod 2.
	\end{align}
	If $p\nmid(4n+3)$, then $c_3\left(\dfrac{n-\frac{3(p-1)}{4}}{p}\right)=0$. Hence, from the above equation, we get
	\begin{align}\label{e1.1.15}
	c_3\left(p^2n+\dfrac{3(p^2-1)}{4}\right)\equiv 0\pmod 2.
	\end{align}
	Replacing $n$ by $pn+\dfrac{3(p-1)}{4}$ in \eqref{e1.1.9b} gives
	\begin{align}\label{e1.1.16}
	c\left(p^3n+\dfrac{3(p^3-1)}{4}\right)\equiv c(n)\pmod 2.
	\end{align}
	By \eqref{e1.1.16} and mathematical induction, for $n\ge 0$ and $k\ge 0$, we have
	\begin{align}
		c_3\left(p^{3k}n+\dfrac{3(p^{3k}-1)}{4}\right)\equiv c_3(n)\pmod 2.\label{e1.1.10}
	\end{align}
	Replacing $n$ by $p^2+\dfrac{p^2-1}{2}$ in \eqref{e1.1.10} and using \eqref{e1.1.15}, we find that if $c_3\left(\dfrac{3(p-1)}{4}\right)\equiv 1\pmod 2$, then for $n\ge 0$ and $k\ge0$ with $p\nmid(4n+3)$, we have
	\begin{align}\label{e1.1.17}
	c_3\left(p^{3k+2}n+\dfrac{3(p^{3k+2}-1)}{4}\right)\equiv 0\pmod 2.
	\end{align}
	From \eqref{e4n}, \eqref{e4n3o}, and \eqref{e1.1.0}, for $n\ge0$, $m,\alpha,\beta\ge1$ with $2\nmid \alpha$, we have
	\begin{align}
		\bar{a}_{2^m\alpha}(4n+3)&\equiv 2^{m+2}c_3(n)\pmod {2^{m+3}}\label{e1.1.12}\\
		\bar{a}_{2\beta+1}(4n+3)&\equiv 8c_3(n)\pmod{16}.\label{e1.1.13}
	\end{align}
	Therefore, the desired result follows from \eqref{e1.1.8}, \eqref{e1.1.12},  and \eqref{e1.1.13}.
\end{proof}
\begin{proof}[Proof of Theorem \ref{t8n4}\unskip]
		From \eqref{e8n4o}, \eqref{e8n4e}, and \eqref{e1.3.1} for $n\ge 0$, $m,\alpha,\beta\ge 1$, we have
	\begin{align}
		\bar{a}_{2\alpha}(8n+4)&\equiv 4c_1(n)\pmod{8}\label{e1.3.19}\\
		\bar{a}_{2\beta+1}(8n+4)&\equiv 2c_1(n)\pmod 4.\label{e1.3.20}
	\end{align}
	Therefore, the desired result follows from \eqref{e1.3.8}, \eqref{e1.3.19} and \eqref{e1.3.20}.
\end{proof}
	\begin{proof}[Proof of Theorem \ref{t8n5}\unskip]
		The proof is identical to the proof of Theorem \ref{t8n7}, hence we omit the details. 
\end{proof}
		\begin{proof}[Proof of Theorem \ref{t8n6}\unskip]
		From \eqref{e8n6} and \eqref{e1.1.0}, for $n\ge 0$ and $s\ge 1$, we have
		\begin{align}\label{e1.1.14}
			\bar{a}_{s}(8n+6)\equiv 8c_3(n)\pmod{16}.
		\end{align}
		Therefore, the desired result follows from \eqref{e1.1.8} and \eqref{e1.1.14}.
	\end{proof}

		\begin{proof}[Proof of Theorem \ref{t8n7}\unskip]
		Thanks to Lemma \ref{lnew62}, from \eqref{e1.1}, we have
		\begin{align}\label{e3.1}
			c_4\displaystyle\left(p^2n+\dfrac{7(p^2-1)}{8}\right)=\gamma(n)c_4(n)-p^9c_4\displaystyle\left(\frac{n-\frac{7(p^2-1)}{8}}{p^2}\right),
		\end{align}
		where
		\begin{align}\label{e3.2}
			\gamma(n)=p^9\alpha-p^4\displaystyle\left(\frac{2n-\frac{7(p^2-1)}{4}}{p}\right)_L.
		\end{align}
		We set $n=0$ in \eqref{e3.1}, to obtain
		\begin{align}\label{e3.3}
			c_4\displaystyle\left(\frac{7(p^2-1)}{8}\right)=\gamma(0).
		\end{align}
		Setting $n=0$ in \eqref{e3.2} and using \eqref{e3.3}, yields
		\begin{align}\label{e3.4}
			p^9\alpha=c_4\displaystyle\left(\frac{7(p^2-1)}{8}\right)+p^4\displaystyle\left(\frac{\frac{-7(p^2-1)}{4}}{p}\right)_L:=\kappa(p).
		\end{align}
		Using \eqref{e3.3} and \eqref{e3.4}, equation \eqref{e3.1} can be rewritten as
		\begin{align}\label{e3.5}
			c_4\displaystyle\left(p^2n+\dfrac{7(p^2-1)}{8}\right)\equiv\displaystyle\left(\kappa(p)-\displaystyle\left(\frac{2n-\frac{7(p^2-1)}{4}}{p}\right)_L\right)c_4(n)+c_4\displaystyle\left(\frac{n-\frac{7(p^2-1)}{8}}{p^2}\right)\pmod 2,
		\end{align}
		where prime $p\ge3$.
		Replacing $n$ by $pn+\frac{7(p^2-1)}{8}$ in \eqref{e3.5} yields
		\begin{align}\label{e3.6}
			c_4\displaystyle\left(p^3n+\frac{7(p^4-1)}{8}\right)\equiv\kappa(p)c_4\displaystyle\left(pn+\frac{7(p^2-1)}{8}\right)+c_4\displaystyle\left(\frac{n}{p}\right)\pmod 2.
		\end{align}
		If $\kappa(p)\equiv0\pmod2$ in \eqref{e3.6}, we have
		\begin{align}\label{e3.7}
			c_4\displaystyle\left(p^3n+\frac{7(p^4-1)}{8}\right)\equiv c_4\displaystyle\left(\frac{n}{p}\right)\pmod2.
		\end{align}
		Replacing $n$ by $pn$ in \eqref{e3.7} gives 
		\begin{align}\label{e3.8}
			c_4\displaystyle\left(p^4n+\frac{7(p^4-1)}{8}\right)\equiv c_4(n)\pmod2.	
		\end{align}
		By mathematical induction on \eqref{e3.8}, we obtain for all $k\ge0$,
		\begin{align}\label{e3.9}
			c_4\displaystyle\left(p^{4k}n+\frac{5(p^{4k}-1)}{24}\right)\equiv c_4(n)\pmod2.	
		\end{align}
		From \eqref{e3.6}, if $\kappa(p)\equiv0\pmod2$ and $p\nmid n$, then we have
		\begin{align}\label{e3.10}
			c_4\displaystyle\left(p^3n+\frac{7(p^4-1)}{8}\right)\equiv0\pmod2.
		\end{align}
		Now, replacing $n$ by $p^3n+\frac{7(p^4-1)}{8}$ in \eqref{e3.9} and employing \eqref{e3.10}, we note that if $\kappa(p)\equiv 0\pmod 2$ and $p\nmid n$, then
		\begin{align}\label{e3.11}
			c_4\displaystyle\left(p^{4k+3}n+\frac{7(p^{4k+4}-1)}{8}\right)\equiv0\pmod2
		\end{align}
		From \eqref{e3.6} with $n$ replaced by $pn$, it follows that if $\kappa(p)\equiv 1\pmod 2$, then
		\begin{align}\label{e3.12}
			c_4\displaystyle\left(p^4n+\frac{7(p^4-1)}{8}\right)\equiv c_4\displaystyle\left(p^2n+\frac{7(p^2-1)}{8}\right)+c_4(n)\pmod 2.
		\end{align}
		Replacing $n$ by $pn+\frac{7(p^2-1)}{8}$ in \eqref{e3.12} gives
		\begin{align}\label{e3.13}
			c_4\displaystyle\left(p^5n+\frac{7(p^6-1)}{8}\right)\equiv c_4\displaystyle\left(p^3n+\frac{7(p^4-1)}{8}\right)+c_4\displaystyle\left(pn+\frac{7(p^2-1)}{8}\right)\pmod 2.
		\end{align}
		From \eqref{e3.6} and \eqref{e3.13}, we deduce that if $\kappa(p)\equiv 1\pmod 2$, then
		\begin{align}\label{e3.14}
			c_4\displaystyle\left(p^5n+\frac{7(p^6-1)}{8}\right)\equiv c_4\displaystyle\left(\frac{n}{p}\right) \pmod 2.
		\end{align}
		If $p\nmid n$, then \eqref{e3.14} yields
		\begin{align}\label{e3.15}
			c_4\displaystyle\left(p^5n+\frac{7(p^6-1)}{8}\right)\equiv0\pmod 2.
		\end{align}
		Replacing $n$ by $pn$ in \eqref{e3.14} gives
		\begin{align}\label{e3.16}
			c_4\displaystyle\left(p^6n+\frac{7(p^6-1)}{8}\right)\equiv c_4(n) \pmod 2.
		\end{align}
		By induction on \eqref{e3.16}, we see that for all $k\ge0$,
		\begin{align}\label{e3.17}
			c_4\displaystyle\left(p^{6k}n+\frac{7(p^{6k}-1)}{8}\right)\equiv c_4(n) \pmod 2.
		\end{align}
		Employing \eqref{e3.15} in \eqref{e3.17} with $n$ replaced by $p^5n+\frac{7(p^6-1)}{8}$, we observe that if $\kappa(p)\equiv 1\pmod 2$ and $p\nmid n$, then
		\begin{align}\label{e3.18}
			c_4\displaystyle\left(p^{6k+5}n+\frac{7(p^{6k+6}-1)}{8}\right)\equiv0\pmod 2.
		\end{align}
		We observe that if $p\nmid (8n+7)$, then $\displaystyle\left(\frac{2n-\frac{7(p^2-1)}{4}}{p}\right)_L\equiv1\pmod 2$ and
		\begin{align}\label{e3.19}
			c_4\displaystyle\left(\frac{n-\frac{7(p^2-1)}{8}}{p^2}\right)=0.
		\end{align}
		From \eqref{e3.5} and \eqref{e3.19}, we see that if $\kappa(p)\equiv 1\pmod 2$ and $p\nmid(8n+7)$, then
		\begin{align}\label{e3.20}
			c_4\displaystyle\left(p^2n+\dfrac{7(p^2-1)}{8}\right)\equiv0\pmod2.
		\end{align}
		Replacing $n$ by $p^2n+\frac{7(p^4-1)}{8}$ in \eqref{e3.17} and using \eqref{e3.20}, we obtain
		\begin{align}\label{e3.21}
			c_4\displaystyle\left(p^{6k+2}n+\frac{7(p^{6k+2}-1)}{8}\right)\equiv0\pmod 2.
		\end{align}
		From \eqref{e1.1}, \eqref{e8n}, and \eqref{e8n7o}, for $n\ge 0$, $m,\alpha,\beta\ge 1$ with $2\nmid \alpha$, we have
		\begin{align}
			\bar{a}_{2^m\alpha}(8n+7)&\equiv 2^{m+3}c_4(n)\pmod{2^{m+4}}\label{e3.22}\\
			\bar{a}_{2\beta+1}(8n+7)&\equiv 64c_4(n)\pmod{128}.\label{e3.23}
		\end{align}
		Thus, the result follows from \eqref{e3.11}, \eqref{e3.18}, \eqref{e3.21}, \eqref{e3.22} and \eqref{e3.23}.
	\end{proof}
		
	\bigskip
	\bigskip
	
	\noindent
	Department of Mathematics\\
	Ramanujan School of Mathematical Sciences\\
	Pondicherry University\\
	Puducherry- 605 014, India.\\

	\noindent Email: \texttt{tthejithamp@pondiuni.ac.in}
	
	\noindent	Email: \texttt{dr.fathima.sn@pondiuni.ac.in} (\Letter)

\begin{thebibliography}{99}
		\bibitem{andrew}Andrews, G. E., and Bachraoui, M. E. (2025). Two-color partitions with evens in one color. arXiv preprint arXiv:2512.23391.
		\bibitem{corteel}Corteel, S., and Lovejoy, J. (2004). Overpartitions. Transactions of the American Mathematical Society, 356(4), 1623-1635.
		\bibitem{gasper}Gasper, G., and Rahman, M. (2011). Basic hypergeometric series (Vol. 96). Cambridge university press.
		\bibitem{gordon}Gordon, B., and Hughes, K. (2006, October). Ramanujan congruences for q (n). In Analytic Number Theory: Proceedings of a Conference Held at Temple University, Philadelphia, May 12–15, 1980 (pp. 333-359). Berlin, Heidelberg: Springer Berlin Heidelberg.
		\bibitem{poq}Hirschhorn, M. D. (2017). The Power of q. Developments in mathematics, 49.
		\bibitem{hirsch}Hirschhorn, M. D., and Sellers, J. A. (2025). A Family of Congruences Modulo 7 for Partitions with Monochromatic Even Parts and Multi--Colored Odd Parts. arXiv preprint arXiv:2507.09752.
		\bibitem{ligozat}Ligozat, G. (1975). Courbes modulaires de genre 1 (No. 43). Société mathématique de France.
			\bibitem{martin}Martin, Y. (1996). Multiplicative $\eta$-quotients. Transactions of the American Mathematical Society, 348(12), 4825-4856.
		\bibitem{newman53}Newman, M. (1953). The coefficients of certain infinite products. Proceedings of the American Mathematical Society, 4(3), 435-439.
		\bibitem{newman55}Newman, M. (1955). An identity for the coefficients of certain modular forms. Journal of the London Mathematical Society, 1(4), 488-493.
		\bibitem{newmf}Newman, M. (1959). Construction and application of a class of modular functions (II). Proceedings of the London Mathematical Society, 3(3), 373-387.
		\bibitem{newman59}Newman, M. (1959). Modular forms whose coefficients possess multiplicative properties. Annals of Mathematics, 70(3), 478-489.
		\bibitem{newman62}Newman, M. (1962). Modular forms whose coefficients possess multiplicative properties, II. Annals of Mathematics, 75(2), 242-250.
		\bibitem{wom}Ono, K. (2004). The Web of Modularity: Arithmetic of the Coefficients of Modular Forms and $ q $-series: Arithmetic of the Coefficients of Modular Forms and Q-series (No. 102). American Mathematical Soc.
	\end{thebibliography}
\end{document}